\documentclass[reqno,11pt,a4paper]{amsart}

\usepackage{amsmath,amsthm,amssymb,mathtools}
\usepackage{multicol,hhline,comment,graphicx,url,color}
\usepackage{mathrsfs}	
\usepackage[margin=1.4in]{geometry}
\usepackage{accents} 
\usepackage[foot]{amsaddr}
\usepackage{hyperref}
\usepackage{cellspace}		
\usepackage{enumitem} 
\setlist[enumerate,1]{leftmargin=0.85cm}
\setlist[itemize,1]{leftmargin=0.55cm}

\usepackage{accents,stackengine,scalerel}

\theoremstyle{plain}

\newtheorem{theorem}{Theorem}[section]

\newtheorem{proposition}[theorem]{Proposition}
\newtheorem{corollary}[theorem]{Corollary}

\theoremstyle{definition}
\newtheorem{definition}[theorem]{Definition}

\theoremstyle{remark}

\newcommand{\ff}{\mathbf{f}}

\makeatletter
 \DeclareRobustCommand{\checkarg}{\@ifnextchar[{\@witharg}{}}
 \DeclareRobustCommand{\@witharg}[1][]{\ensuremath{\left(#1\right)}}
 
 \DeclareRobustCommand{\scaleGen}[1]{\@ifnextchar[{\@scalewithargs{#1}}{\odot^{}_{#1}}}
 \def\@scalewithargs#1[#2][#3]{#2 \odot^{}_{#1} #3}
\makeatother

\def\Exc{\mathcal{E}}

\def\IPmag#1{\left\|\vphantom{I}#1\right\|}		

\def\skewer{\textsc{skewer}}		
\def\skewerP{\widebar{\skewer}}		

\def\bN{\mathbf{N}}			
\def\bF{\mathbf{F}}			
\def\bG{\mathbf{G}}			
\def\bX{\mathbf{X}}			

\def\BR{\mathbb{R}}				

\def\distribfont#1{\texttt{\upshape #1}}

\def\PRM{\distribfont{PRM}\checkarg}

\def\Stable{\distribfont{Stable}\checkarg}
\def\BESQ{\distribfont{BESQ}\checkarg}

\makeatletter
\let\save@mathaccent\mathaccent
\newcommand*\if@single[3]{%
  \setbox0\hbox{${\mathaccent"0362{#1}}^H$}%
  \setbox2\hbox{${\mathaccent"0362{\kern0pt#1}}^H$}%
  \ifdim\ht0=\ht2 #3\else #2\fi
  }
\newcommand*\rel@kern[1]{\kern#1\dimexpr\macc@kerna}
\newcommand{\widebar}{}
\DeclareRobustCommand*\widebar[1]{\@ifnextchar^{\wide@bar{#1}{0}}{\wide@bar{#1}{1}}}
\newcommand*\wide@bar[2]{\if@single{#1}{\wide@bar@{#1}{#2}{1}}{\wide@bar@{#1}{#2}{2}}}
\newcommand*\wide@bar@[3]{%
  \begingroup
  \def\mathaccent##1##2{%
    \let\mathaccent\save@mathaccent
    \if#32 \let\macc@nucleus\first@char \fi
    \setbox\z@\hbox{$\macc@style{\macc@nucleus}_{}$}%
    \setbox\tw@\hbox{$\macc@style{\macc@nucleus}{}_{}$}%
    \dimen@\wd\tw@
    \advance\dimen@-\wd\z@
    \divide\dimen@ 3
    \@tempdima\wd\tw@
    \advance\@tempdima-\scriptspace
    \divide\@tempdima 10
    \advance\dimen@-\@tempdima
    \ifdim\dimen@>\z@ \dimen@0pt\fi
    \rel@kern{0.6}\kern-\dimen@
    \if#31
      \overline{\rel@kern{-0.6}\kern\dimen@\macc@nucleus\rel@kern{0.4}\kern\dimen@}%
      \advance\dimen@0.4\dimexpr\macc@kerna
      \let\final@kern#2%
      \ifdim\dimen@<\z@ \let\final@kern1\fi
      \if\final@kern1 \kern-\dimen@\fi
    \else
      \overline{\rel@kern{-0.6}\kern\dimen@#1}%
    \fi
  }%
  \macc@depth\@ne
  \let\math@bgroup\@empty \let\math@egroup\macc@set@skewchar
  \mathsurround\z@ \frozen@everymath{\mathgroup\macc@group\relax}%
  \macc@set@skewchar\relax
  \let\mathaccentV\macc@nested@a
  \if#31
    \macc@nested@a\relax111{#1}%
  \else
    \def\gobble@till@marker##1\endmarker{}%
    \futurelet\first@char\gobble@till@marker#1\endmarker
    \ifcat\noexpand\first@char A\else
      \def\first@char{}%
    \fi
    \macc@nested@a\relax111{\first@char}%
  \fi
  \endgroup
}
\makeatother

\numberwithin{equation}{section}
\numberwithin{figure}{section}
\numberwithin{table}{section}

\allowdisplaybreaks[3]

\begin{document}

 
 \ \vspace{-22pt}
 
 \title[Interval partitions from Bertoin's study of ${\tt BES}_0(\lowercase{d})$]{Diffusions on a space of interval partitions:\\ construction from Bertoin's ${\tt BES}_0(\lowercase{d})$, $\lowercase{d}\in(0,1)$}
 
 
  \author[M.~Winkel]{Matthias Winkel}
  \address{Department of Statistics\\ University of Oxford\\ Oxford OX1 3LB\\ UK}
  \email{winkel@stats.ox.ac.uk}
  

 %
 %
  
\begin{abstract} 
  In 1990, Bertoin constructed a measure-valued Markov process in the framework of a Bessel process of dimension between 0 and 1. 
  In the present paper, we represent this process in a space of interval partitions. 
  We show that this is a member of a class of interval partition diffusions introduced recently and independently by Forman, Pal, Rizzolo and Winkel
  using a completely different construction from spectrally positive stable L\'evy processes with index between 1 and 2 and with jumps marked by 
  squared Bessel excursions of a corresponding dimension between $-2$ and 0. 
\end{abstract}  
 

\maketitle

 \ \vspace{-22pt}

\section{Introduction}
\label{sec:intro}

\noindent We define interval partitions, following Aldous \cite[Section 17]{AldousExch} and Pitman \cite[Chapter 4]{CSP}.

\begin{definition}\label{def:IP_1}
 An \emph{interval partition} is a set $\beta$ of disjoint, open subintervals of some finite real interval $[0,M]$, that cover $[0,M]$ up to a Lebesgue-null set. We write $\IPmag{\beta}$ to denote $M$. We refer to the elements of an interval partition as its \emph{blocks}. The Lebesgue measure of a block is called its \emph{mass} or \emph{size}.
\end{definition}

In this paper we construct diffusion processes in a space of interval partitions in Bertoin's \cite{Bertoin1990,Bertoin1990c} framework of 
a Bessel process of dimension $d\in(0,1)$. Bertoin studied the excursions of such a Bessel process. Specifically, he first 
decomposed the Bessel process \vspace{-0.1cm}
\begin{equation}\label{framework}\mathbf{R}=\mathbf{B}-(1-d)\mathbf{H}
\end{equation}
into a Brownian motion $\mathbf{B}$ and a path-continuous process 
$\mathbf{H}$ with zero quadratic variation. He constructed excursions of the Markov process $(\mathbf{R},\mathbf{H})$ away from $(0,0)$, each consisting of infinitely many excursions of $\mathbf{R}$ away from $0$.
By extracting suitable statistics, namely the set $\{\mathbf{R}(t)\colon t\ge 0,\,\mathbf{H}(t)=y\}$, he showed \cite[Theorems II.2--II.3]{Bertoin1990c} that the 
measure-valued process\vspace{-0.1cm}
  \begin{equation}\label{mvp} y\mapsto\mu^y_{[0,T]}:=\sum_{0\le t\le T\colon\mathbf{H}(t)=y,\mathbf{R}(t)\neq 0}\delta_{\mathbf{R}(t)}\vspace{-0.1cm}
  \end{equation} 
is path-continuous (in the vague topology for sigma-finite point measures on $(0,\infty)$) and Markovian when $T$ is chosen suitably such 
as an inverse local time $\tau_{\mathbf{R},\mathbf{H}}^0(u)$, $u\ge 0$, of $(\mathbf{R},\mathbf{H})$ at $(0,0)$. He further
showed in \cite[Theorem 4.2]{Bertoin1990} and \cite[Corollary II.4]{Bertoin1990c} that\vspace{-0.2cm} 
\begin{equation}\label{eq:loctime}y\mapsto\lambda^y(T):=2\int_{(0,\infty)}x\mu_{[0,T]}^y(dx)=2\sum_{0\le t\le T\colon\mathbf{H}(t)=y}\mathbf{R}(t)
\end{equation}
is \BESQ[0], a zero-dimensional squared Bessel process. We provide a more comprehensive review of Bertoin's results in Section \ref{sec:Bertoin}. 
In this paper, we represent his measure-valued process \eqref{mvp} as a diffusion in a space of interval partitions.

\begin{theorem}\label{thm:Bertoin} In the setting of \eqref{framework}--\eqref{eq:loctime}, with $T=\tau_{\mathbf{R},\mathbf{H}}^0(u)$, the interval partitions
  $$\beta^y:=\Big\{\big(\lambda^y(t-),\lambda^y(t)\big)\colon t\in[0,T],\,\mathbf{R}(t)\neq 0,\,\mathbf{H}(t)=y\Big\},\quad y\ge 0,$$
  form a diffusion process in a suitable space interval partitions.
\end{theorem}

While the interval lengths $\lambda^y(t)-\lambda^y(t-)$ of $\beta^y$ are (twice) the locations $\mathbf{R}(t)$ of atoms of $\mu_{[0,T]}^y$, 
the order of the intervals is not captured by $\mu_{[0,T]}^y$. Hence, this theorem is not an immediate consequence of Bertoin's corresponding
results for $(\mu_{[0,T]}^y,\,y\ge 0)$

Indeed, we prove this theorem by identifying this diffusion process as an instance of a class of diffusion processes introduced in \cite{IPPA}, where we gave a general construction of processes in a space of interval partitions based on spectrally positive L\'evy processes
(\em scaffolding\em) whose point process of jump heights (interpreted as lifetimes of individuals) is marked by excursions (\em spindles\em, giving ``sizes'' varying during the 
lifetime, one for each level crossed). Informally, the interval partition evolution, indexed by level, considers for each level $y\ge 0$
the jumps crossing that level and records for each such jump an interval whose length is the ``size'' of the individual (width of the spindle) when crossing that level, 
ordered from left to right without leaving gaps. This construction and terminology is illustrated in Figure \ref{fig:skewer_1}.

\begin{figure}[t]
 \centering
 \begin{picture}(0,0)%
\includegraphics[scale=0.92]{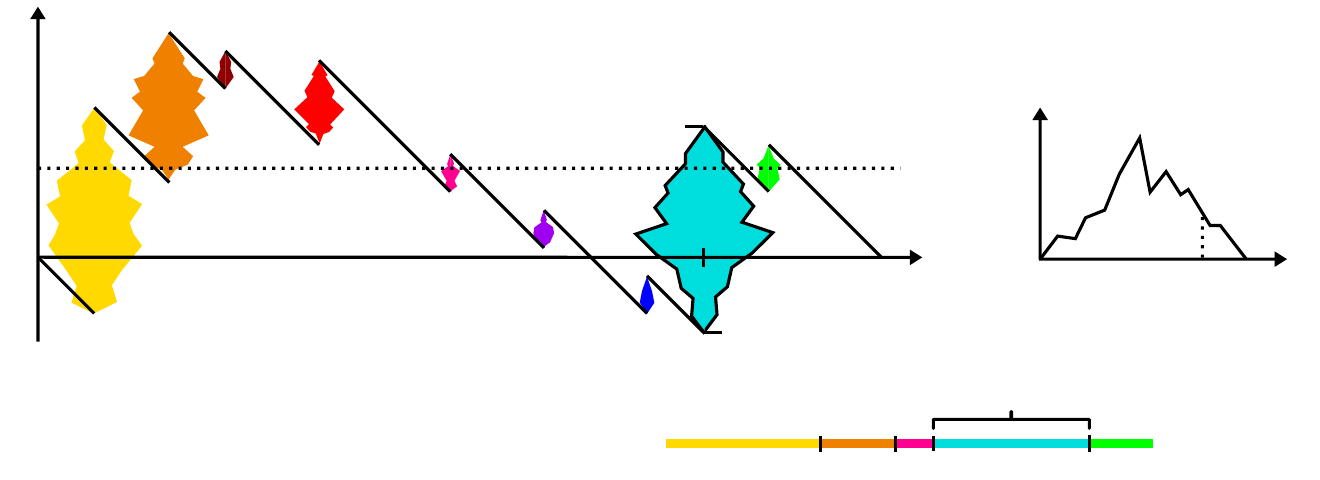}%
\end{picture}%
\setlength{\unitlength}{3730sp}%
\begingroup\makeatletter\ifx\SetFigFont\undefined%
\gdef\SetFigFont#1#2#3#4#5{%
  \reset@font\fontsize{#1}{#2pt}%
  \fontfamily{#3}\fontseries{#4}\fontshape{#5}%
  \selectfont}%
\fi\endgroup%
\begin{picture}(6324,2218)(304,-2699)
\put(3600,-1574){\makebox(0,0)[b]{\smash{{\SetFigFont{10}{12.0}{\familydefault}{\mddefault}{\updefault}{\color[rgb]{0,0,0}$s_j$}%
}}}}
\put(4451,-736){\makebox(0,0)[rb]{\smash{{\SetFigFont{10}{12.0}{\familydefault}{\mddefault}{\updefault}{\color[rgb]{0,0,0}$(N,X)$}%
}}}}
\put(5076,-2327){\rotatebox{360.0}{\makebox(0,0)[b]{\smash{{\SetFigFont{11}{13.2}{\familydefault}{\mddefault}{\updefault}{\color[rgb]{0,0,0}$f_j(y-X(s_j-))$}%
}}}}}
\put(406,-1276){\makebox(0,0)[rb]{\smash{{\SetFigFont{11}{13.2}{\familydefault}{\mddefault}{\updefault}{\color[rgb]{0,0,0}$y$}%
}}}}
\put(3411,-1096){\makebox(0,0)[rb]{\smash{{\SetFigFont{11}{13.2}{\familydefault}{\mddefault}{\updefault}{\color[rgb]{0,0,0}$X(s_j)$}%
}}}}
\put(3731,-2086){\makebox(0,0)[lb]{\smash{{\SetFigFont{11}{13.2}{\familydefault}{\mddefault}{\updefault}{\color[rgb]{0,0,0}$X(s_j-)$}%
}}}}
\put(6400,-1733){\makebox(0,0)[lb]{\smash{{\SetFigFont{10}{12.0}{\familydefault}{\mddefault}{\updefault}{\color[rgb]{0,0,0}$z$}%
}}}}
\put(5775,-1000){\makebox(0,0)[b]{\smash{{\SetFigFont{10}{12.0}{\familydefault}{\mddefault}{\updefault}{\color[rgb]{0,0,0}$(f_j(z),\ z\geq 0)$}%
}}}}
\put(3231,-2600){\makebox(0,0)[rb]{\smash{{\SetFigFont{11}{13.2}{\familydefault}{\mddefault}{\updefault}{\color[rgb]{0,0,0}$\skewer(y,N,X)$}%
}}}}
\end{picture}\vspace{-0.2cm}%
 \caption{Left: The sloping black lines show the scaffolding $X$. Shaded blobs decorating jumps show the corresponding spindles: points $(s_j,f_j)$ of $N$. Right: Graph of one spindle. Bottom: A skewer, with blocks shaded to correspond to spindles; not to scale.\label{fig:skewer_1}\vspace{-0.2cm}}
\end{figure}
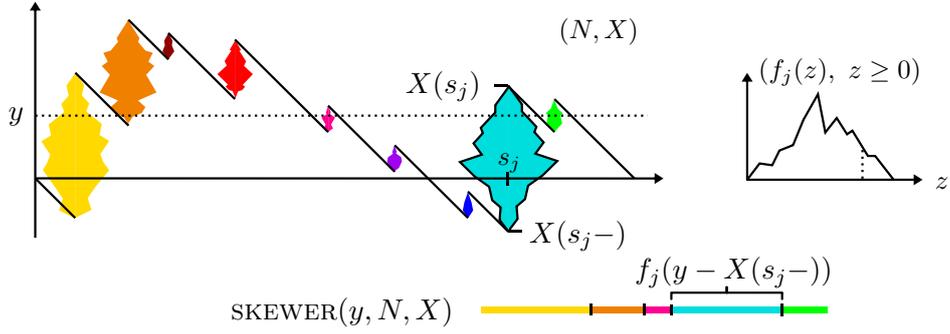

Specifically, if $N=\sum_{i\in I}\delta_{(s_i,f_i)}$ is a point process of times $s_i\in[0,S]$ and excursions $f_i$ of excursion lengths $\zeta_i$ (spindle heights), 
and $X$ is a real-valued process with jumps $\Delta X(s_i):=X(s_i)-X(s_i-)=\zeta_i$ at times $s_i$, $i\in I$, we define the interval partition $\skewer(y,N,X)$ at 
level $y$, as follows.

\begin{definition}\label{def:skewer}
 For $y\in\BR$, $s\in [0,S]$, the \emph{aggregate mass} in $(N,X)$ at level $y$, up to time $s$ is
 \begin{equation}
  M_{N,X}^y(s) := \sum_{i\in I\colon s_i\le s}f_i(y - X(s_i-)).\label{eq:agg_mass_from_spindles}
 \end{equation}
 The \emph{skewer} of $(N,X)$ at level $y$, denoted by $\skewer(y,N,X)$, is defined as
 \begin{equation}
   \left\{\left(M^y_{N,X}(s-),M^y_{N,X}(s)\right)\!\colon s\in [0,S],\,M^y_{N,X}(s-) < M^y_{N,X}(s)\right\}\label{eq:skewer_def}
 \end{equation}
 and the \emph{skewer process} as $\skewerP(N,X) := \big( \skewer(y,N,X),y\!\geq\! 0\big)$. 
\end{definition}

This definition is meaningful when $X$ has finitely many jumps as in Figure \ref{fig:skewer_1}, and also when $X$ has a dense set of jump times and the $f_i$ are such that $M_{N,X}$ is finite. In \cite{IPPA}, we established criteria under which $\skewerP(\bN,\bX)$ is a diffusion. Specifically, $\bN$ is a Poisson random measure ({\tt PRM}) with intensity 
measure ${\tt Leb}\otimes\nu$, where $\nu$ is the Pitman--Yor excursion law \cite{PitmYor82} associated with a suitable (self-similar) $[0,\infty)$-valued diffusion, 
and $\bX$ is an associated L\'evy process, suitably stopped at a time $\mathbf{S}$ when $\bX$ is zero. In 
this interval partition evolution, each interval length (block) evolves independently according to the 
$[0,\infty)$-valued diffusion, which we call \em block diffusion\em, while between (the infinitely many) blocks, new blocks appear at the pre-jump levels of $\bX$. The
{\tt PRM} of jumps is obtained by mapping the {\tt PRM} of spindles onto the spindle heights. Conversely, we may view the {\tt PRM} of spindles as marking the
{\tt PRM} of jumps by \em block excursions. \em See Section \ref{sec:prel} for more details. 


\begin{theorem}\label{thm:diffusion_0} When the block diffusion is \BESQ[-2(1-d)], a squared Bessel process of dimension $-2(1\!-\!d)\!\in\!(-2,0)$ and 
  the scaffolding L\'evy process is \Stable[2\!-\!d] stopped at an inverse local time $\tau_\mathbf{X}^0(v)$ of $\mathbf{X}$ at 0, the interval partition evolution associated via $\skewerP$ is distributed as the diffusion in Theorem 
  \ref{thm:Bertoin}, for $u=2^{d-1}v$. 
\end{theorem}

The remainder of this paper is organised, as follows. In Sections \ref{sec:Bertoin} and \ref{sec:prel}, we state the main results of 
 \cite{Bertoin1990,Bertoin1990c} and \cite{IPPA,IPPB}, exhibiting the parallels. In Section \ref{sec:new}, we make precise the connections between
the two frameworks and deduce the theorems we have stated. In Section \ref{sec:new2}, we discuss some further observations.

\section{Bertoin's study of Bessel processes \cite{Bertoin1990,Bertoin1990c}}\label{sec:Bertoin}



\noindent Consider a Bessel process $\mathbf{R}\sim{\tt BES}_x(d)$ of dimension $d\in(0,1)$ starting from $x>0$. Let $T_\mathbf{R}(0)=\inf\{t\ge 0\colon\mathbf{R}(t)=0\}$. On $[0,T_\mathbf{R}(0))$, the Bessel process
$\mathbf{R}$ satisfies an SDE that yields 
\begin{equation}\mathbf{R}(t)=x+\mathbf{B}(t)-\frac{1-d}{2}\int_0^t\frac{du}{\mathbf{R}(u)}.\label{bessde}
\end{equation}
Furthermore, this singular integral is finite as $t\uparrow T_\mathbf{R}(0)$. By time reversal, this means that this integral is also well-defined under the 
excursion measure of the Bessel process. While the excursions can be stitched together to form a Bessel process that has 0 as a 
reflecting boundary, the positive values of these integrals are not summable so that the representation (\ref{bessde}) fails beyond $T_\mathbf{R}(0)$. However,
(\ref{bessde}) can be extended beyond $T_\mathbf{R}(0)$ if some compensation is introduced, as follows. It is well-known that the Bessel process $\mathbf{R}$ has 
jointly continuous space-time local times on $(0,\infty)^2$. To obtain a family of local times that extends continuously to $[0,\infty)^2$, it is 
convenient to choose the level-$a$ local time $(L^a(t),t\ge 0)$, $a>0$, of $\mathbf{R}$ such that the occupation density of $\mathbf{R}$ is $(a^{d-1}L^a(t),a>0,t\ge 0)$. By the occupation density formula and since $L^0(t)=0$ for 
$t<T_\mathbf{R}(0)$, we can write
\begin{equation}\label{eq:defH}\frac{1}{2}\int_0^t\frac{du}{\mathbf{R}(u)}=\frac{1}{2}\int_0^\infty a^{d-1}L^a(t)\frac{da}{a}=\frac{1}{2}\int_0^\infty a^{d-2}(L^a(t)-L^0(t))da=:\mathbf{H}(t)
\end{equation}
  for $t<T_\mathbf{R}(0)$. 
Bertoin showed that defining $\mathbf{H}(t)$ by the right-most integral in \eqref{eq:defH} also for $t\ge T_\mathbf{R}(0)$ yields a path-continuous process $\mathbf{H}$ with unbounded variation, but 
zero quadratic variation (the finiteness of $\mathbf{H}$ follows from the H\"older continuity of $L^a(t)$ in $a$).
Clearly, this process $\mathbf{H}$ is increasing on all excursion intervals of $\mathbf{R}$ away from zero, but the effect of the compensating local time at zero 
is that $\mathbf{H}$ does not increase across the zero-set of $\mathbf{R}$. With this notation, we have
\begin{equation}\mathbf{R}(t)=x+\mathbf{B}(t)-(1-d)\mathbf{H}(t),\qquad \mbox{for all }t\ge 0.\label{besdec}
\end{equation}
Bertoin noted that $(\mathbf{R},\mathbf{H})$ is a Markov process and that $(0,0)$ is recurrent for this Markov process. It is instructive to 
consider the excursions of $(\mathbf{R},\mathbf{H})$ away from $(0,0)$ by plotting $\mathbf{R}(t)$ against ``time'' $\mathbf{H}(t)$. Since $\mathbf{H}$ 
increases during each excursion of $\mathbf{R}$ away from 0, on $(\ell_s,r_s)$, say, such a plot shows a time-changed excursion of $\mathbf{R}$ 
starting from 0 at ``time'' $\mathbf{H}(\ell_s)$ and returning to 0 at ``time'' $\mathbf{H}(r_s)>\mathbf{H}(\ell_s)$. As $\mathbf{H}$ does not 
increase across the zero-set of $\mathbf{R}$, the excursions for different $(\ell_s,r_s)$ overlap, in general, when included in the same plot.

Since $\mathbf{H}$ is increasing when $\mathbf{R}$ is away from 0, and can only decrease across the zero-set of $\mathbf{R}$, the excursions of 
$(\mathbf{R},\mathbf{H})$ away from $(0,0)$ typically consist of many excursions of $\mathbf{R}$. Specifically, each excursion of 
$(\mathbf{R},\mathbf{H})$ can be decomposed into three parts: first, at the ``beginning'', there is an escape from $(0,0)$ towards the 
left by an accumulation of short $\mathbf{R}$-excursions until, in the ``middle'', one $\mathbf{R}$-excursion takes $(\mathbf{R},\mathbf{H})$ 
across to positive $\mathbf{H}$-values and, at the ``end'', there is a final approach back to $(0,0)$ from the right by an accumulation of short $\mathbf{R}$-excursions.\medskip

\noindent The main objects of interest in Bertoin's work \cite{Bertoin1990,Bertoin1990c} are
\begin{itemize}\item the excursions away from $(0,0)$ of $(\mathbf{R},\mathbf{H})$, and associated quantities,
  \item the excursions away from 0 of $\widetilde{\mathbf{R}}\!:=\!2\mathbf{R}\circ T_\mathbf{H}^+$, where $T_\mathbf{H}^+(y)\!=\!\inf\{t\!\ge\! 0\colon\!\mathbf{H}(t)\!>\!y\}$,\vspace{0.1cm}
  \item local time processes $(\lambda^y(t),y\ge 0)$ and $(\lambda^{-y}(t),y\ge 0)$ of $\mathbf{H}$, up to time $t\ge 0$,\vspace{0.1cm}
  \item measure-valued processes $y\!\mapsto\!\mu_{[0,T]}^y\!=\!\sum_{0\le t\le T\colon\mathbf{H}(t)=y,\mathbf{R}(t)\neq 0}\delta_{\mathbf{R}(t)}$ for some $T\!>\!0$.
\end{itemize}
Specifically, some of the main results of \cite{Bertoin1990} are the following. We use Bertoin's numbering for
ease of reference.

\begin{enumerate}
  \item[2.4] The inverse local time $\tau_{\mathbf{R},\mathbf{H}}^{(0,0)}$ of $(\mathbf{R},\mathbf{H})$ at $(0,0)$ is stable with index $(1-d)/2$.
  \item[3.1] The It\^o excursion process of $(\mathbf{R},\mathbf{H})$ is a \PRM.\vspace{0.1cm}
  \item[3.2] Under the excursion measure, excursions of $(\mathbf{R},\mathbf{H})$ are time-reversible.\vspace{0.1cm}
  \item[3.3]\begin{enumerate}
                    \item[(i)] A.s., all excursions of $(\mathbf{R},\mathbf{H})$ away from $(0,0)$ start into $[0,\infty)\!\times\!(-\infty,0)$, 
                                                                         cross $(0,\infty)\times\{0\}$ at a unique time $U$
                                                                         and finish from $[0,\infty)\times(0,\infty)$. 
    	         \item[(ii)] The \PRM\ has points at excursions whose value of $\mathbf{R}$ when the excursion is crossing the line $\mathbf{H}=0$ 
    	         							    has (sigma-finite) law $((1-d)/\Gamma(d))x^{d-2}dx$, and
    	         \item[(iii)] points at excursions with an $\mathbf{H}$-infimum below $-y$ occur at rate $y^{d-1}$.\vspace{0.1cm}
    	       \end{enumerate}  
  \item[3.4] Mid-excursion Markov property: conditionally given an $\mathbf{R}$-value of $\mathbf{R}(U)=x$ at the crossing time $U$ of the line $\mathbf{H}=0$, the 
    post-$U$ part of the excursion and the time-reversed pre-$U$ part are independent and distributed as the process $(\mathbf{R},\mathbf{H})$ starting from $(x,0)$ and
    stopped when hitting $(0,0)$.\vspace{0.1cm}
  \item[4.1] In the setting of result 3.4, conditionally given $\mathbf{R}(U)\!=\!x$, the above-0 and below-0 local 
    time processes  $(\lambda^y(t),y\ge 0)$ and $(\lambda^{-y}(t),y\ge 0)$ of $\mathbf{H}$ during an excursion of $(\mathbf{R},\mathbf{H})$ are two independent $\BESQ_{2x}(0)$.
  \item[4.2] The level-0 local time $\lambda^0$ of $\mathbf{H}$ time-changed by the inverse local time $\tau_{\mathbf{R},\mathbf{H}}^{(0,0)}$ is a 
    stable subordinator of index $1-d$. Given $\lambda^0(\tau_{\mathbf{R},\mathbf{H}}^{(0,0)}(u))\!=\!x$, the processes $(\lambda^y(\tau_{\mathbf{R},\mathbf{H}}^{(0,0)}(u)),y\!\ge\! 0)$ and
    $(\lambda^{-y}(\tau_{\mathbf{R},\mathbf{H}}^{(0,0)}(u)),y\!\ge\! 0)$ are two independent $\BESQ_x(0)$.
\end{enumerate}
The main additional results of \cite{Bertoin1990c} are the following.

\begin{enumerate}
  \item[I.5] Under the It\^o excursion measure of $\widetilde{\mathbf{R}}$, excursions
    \begin{enumerate}\item[(i)] start positive with initial values at rate $(2^{1-d}(1\!-\!d)/\Gamma(d))x^{d-2}dx$, and
                                   \item[(ii)] when starting from $x$ evolve as $\BESQ_x(-2(1-d))$.\vspace{0.1cm}
    \end{enumerate}  
  \item[I.6] The semi-group of $\widetilde{\mathbf{R}}$ is characterised by its Laplace transforms, for $\gamma\ge 0$,
    $\mathbb{E}_x(\exp(-\gamma\widetilde{\mathbf{R}}(y)))=\exp(-x/2y)\left((1+2\gamma y)^{1-d}\exp(x/(2+4\gamma y))-(2\gamma y)^{1-d}\right)$.
  \item[II.1] The measure-valued process $y\!\mapsto\!\mu^y_{[0,T_{\mathbf{H}}(-1)]}$ for $T_{\mathbf{H}}(-1)\!:=\!\inf\{t\!\ge\! 0\colon\mathbf{H}(t)\!=\!-1\}$ 
    admits a continuous version in the space $\mathcal{N}((0,\infty))$ of point measures that are finite on $(\varepsilon,\infty)$ for all $\varepsilon>0$, 
    equipped with the topology of vague convergence.
  \item[II.2]
    \begin{enumerate}
      \item[(i)] The process $(\mu^y_{[0,T_{\mathbf{H}}(-1)]},y\ge 0)$ is Markovian.
      \item[(ii)] Its semi-group $\kappa_y^{\mathcal{N}}$, $y\!\ge\! 0$, acts on functions $f_\varphi(\sum_{i\in I}n_i\delta_{x_i})\!=\!\prod_{i\in I}(\varphi(x_i))^{n_i}$ for 
        continuous $\varphi\colon(0,\infty)\rightarrow[0,1]$ as $\kappa_y^{\mathcal{N}}f_\varphi=f_{\varphi_y}$, where $\varphi_y(x)$ is given by 
        $e^{-x/y}\!+\!\int_0^\infty\varphi(a)p_y(x,da)\big/\big(1+y^{1-d}\!\int_0^\infty((1\!-\!d)/\Gamma(d))s^{d-2}(1\!-\!\varphi(s))e^{-s/y}ds\big)$, where
        $\int_0^\infty e^{-\gamma a}p_y(x,da)=(1+\gamma y)^{1-d}\big(e^{-\gamma x/(1+\gamma y)}-e^{-x/y}\big)$, for all $\gamma\ge 0$.
      \item[(iii)] The process $(\mu^{-1+y}_{[0,T_{\mathbf{H}}(-1)]},0\!\le\! y\!\le\! 1)$ is Markovian with semi-group $\widetilde{\kappa}_y^{\mathcal{N}}$, $y\!\ge\! 0$, given by 
        $\widetilde{\kappa}_y^{\mathcal{N}}f_\varphi=f_{\varphi_y}\big/\big(1+y^{1-d}\!\int_0^\infty((1\!-\!d)/\Gamma(d))s^{d-2}(1\!-\!\varphi(s))e^{-s/y}ds\big)$.
    \end{enumerate}
  \item[II.3] Given $\mu^0_{[0,\tau_{\mathbf{R},\mathbf{H}}^{(0,0)}(1)]}$, the processes $(\mu^y_{[0,\tau_{\mathbf{R},\mathbf{H}}^{(0,0)}(1)]},y\!\ge\! 0)$ and $(\mu^{-y}_{[0,\tau_{\mathbf{R},\mathbf{H}}^{(0,0)}(1)]},y\!\ge\! 0)$ are conditionally independent and have the semi-group $(\kappa_y^\mathcal{N},y\!\ge\! 0)$, of II.2(ii). 
  \item[II.4]
    \begin{enumerate}
      \item[(i)] The process $(\lambda^{-1+y}(T_\mathbf{H}(-1)),0\!\le\! y\!\le\! 1)$ is a ${\tt BESQ}_0(2-2d)$.
      \item[(ii)] Given $\lambda^0(T_\mathbf{H}(-1))=x$, the process $(\lambda^y(T_\mathbf{H}(-1)),y\!\ge\! 0)$ is a ${\tt BESQ}_x(0)$. 
      \item[(iii)]  Given $\lambda^0(\tau_{\mathbf{R},\mathbf{H}}^{(0,0)}(1))=x$, the process $(\lambda^y(\tau_{\mathbf{R},\mathbf{H}}^{(0,0)}(1)),y\!\ge\! 0)$ is a ${\tt BESQ}_x(0)$. \vspace{-0.1cm}
    \end{enumerate}  
\end{enumerate}

\section{Skewer processes of marked L\'evy processes \cite{IPPA,IPPB}\vspace{-0.1cm}}
\label{sec:prel}

\noindent Let $\alpha\!\in\!(0,1)$ 
and $\bX$ a spectrally positive \Stable[1\!+\!\alpha]-process with Laplace exponent $\psi(c)=c^{1+\alpha}/2^\alpha\Gamma(1\!+\!\alpha)$. We call $\bX$ \em scaffolding \em and proceed to decorate it. Specifically, consider the \PRM\ $\sum_{i\in\mathbf{I}}\delta_{(\mathbf{s}_i,\Delta\mathbf{X}(\mathbf{s}_i))}$ of its jumps. For each jump $\Delta\mathbf{X}(\mathbf{s}_i)$, consider an independent $\BESQ(-2\alpha)$ excursion (\em spindle\em) $\mathbf{f}_i$ 
of length $\zeta(\mathbf{f}_i)=\Delta\mathbf{X}(\mathbf{s}_i)$. These excursions were studied by Pitman and Yor \cite{PitmYor82}, who also noted, in their Remark (5.8) on pp. 453f., that when
conditioned on their length, they are $\BESQ(4+2\alpha)$ bridges from 0 to 0. By standard marking of \PRM s, $\bN:=\sum_{i\in\mathbf{I}}\delta_{(\mathbf{s}_i,\mathbf{f}_i)}$ is itself a \PRM\ on the space $[0,\infty)\times\Exc$, 
where $\Exc$ is the space of (continuous) excursion paths.
This is illustrated in a simplified way in Figure \ref{fig:skewer_1}. The intensity measure ${\tt Leb}\otimes\nu$ of $\mathbf{N}$ is the Pitman--Yor excursion measure of
\cite{PitmYor82}, which can be described by entrance laws and a further evolution as unconditioned $\BESQ(-2\alpha)$ processes. 

Recall that the skewer of Definition \ref{def:skewer} extracts from $N=\sum_{i\in I}\delta_{(s_i,f_i)}$ 
all level-$y$ spindle masses $f_i(y-X(s_i))$, where $y\in(X(s_i-),X(s_i))$, $i\in I$, and builds the interval partition that has these as interval lengths in the order given by
the $s_i$, $i\in I$. The set $\mathcal{I}_H$ of all interval partitions can be equipped with a distance $d_H$ that applies the Hausdorff metric to the set of points not
covered by the intervals. 

Since $\bX$ is spectrally positive, its (c\`adl\`ag) excursions away from 0 (or any other level $y$) start negative, jump across zero and end positive. 
Applied to $(\bN,\bX)$, the skewer at level $y$ extracts one block from each excursion of $\bX$ away from $y$. 
In \cite{IPPA}, we denote the \PRM\ of excursions of $\bX$ away from $y$ by $\bG^y$ and enhance the excursion theory of $\bX$ to include $\bN$:
each excursion $e_{[\ell,r]}\!:=\!(-y\!+\!\bX|_{[\ell,r]}(\ell\!+\!s),s\!\in\![0,r\!-\!\ell])$ of $\bX$ has its jumps marked by spindles. 
We denote  by $\bF^y$ the associated random measure whose points are pairs of $e_{[\ell,r]}$ and the restriction $\bN|_{[\ell,r]\times\Exc}$ shifted to $[0,r-\ell]\times\Exc$. In each excursion with spindle marks, the \em central spindle \em crossing 0 can be viewed as the ``middle'' of 
three parts, separating the spindles of the ``beginning'' where $\bX$ is negative from the spindles of the ``end'' where $\bX$ is positive.

We refer to excursions of $(\bX,\bN)$ as \em bi-clades\em, 
to the negative part of such an excursion including the central spindle up to level 0 as an \em anti-clade\em, 
and to the remainder as a \em clade\em. 
To start an interval-partition-valued process from any interval partition $\beta$, we consider \em clades starting from ${\tt Leb}(V)$\em, $V\!\in\!\beta$, as follows.
For each \pagebreak interval $V\!\in\!\beta$ independently, consider $\ff_V\!\sim\!\BESQ_{{\tt Leb}(V)}(-2\alpha)$ and an independent $(\bX,\bN)$ stopped at $S_\mathbf{X}(-\zeta(\ff_V))\!:=\!\inf\{s\!\ge\! 0\colon\bX(s)\!=\!-\zeta(\ff_V)\}$, for the length $\zeta(\ff_V)$ of $\ff_V$, then form the clade 
$(\bX_V,\mathbf{N}_V)\!:=\!(\zeta(\ff_V)\!+\!\bX|_{[0,S_\mathbf{X}(-\zeta(\ff_V))]},\delta_{(0,\ff_V)}\!+\!\bN|_{[0,S_\mathbf{X}(-\zeta(\ff_V))]\times\Exc})$. 
We stitch together all excursions $\bX_V$ in the left-to-right order of $V\!\in\!\beta$ to form a scaffolding $\mathbf{X}_\beta$,
similarly build $\mathbf{N}_\beta$ from $\mathbf{N}_V$, $V\!\in\!\beta$, and consider $\skewerP(\mathbf{N}_\beta,\mathbf{X}_\beta)$.

\medskip

\noindent Some of the main objects of interest are 

\begin{itemize}\item the pair $(\bX,\bN)$ of the \Stable[1+\alpha] scaffolding $\bX$ and the \PRM\ $\bN$ of spindles,
  \item the random point measures $\bF^y$, $y\ge 0$, of bi-clades of $(\bX,\bN)$,
  \item the \em type-1 evolution \em $(\beta^y,y\!\ge\! 0):=\skewerP(\bN_\beta,\bX_\beta)$, extracting intervals from the spindles in
    jumps of $\mathbf{X}_\beta$ crossing level $y$, for any initial interval partition $\beta$.
  \item the total mass process $(\IPmag{\beta^y},\,y\ge 0)$.
\end{itemize}

\noindent Some of the main results of \cite{IPPA} are the following, in the numbering of \cite{IPPA}.

\begin{enumerate}
  \item[1.3] The interval-partition-valued process $y\mapsto\beta^y$ admits a continuous version.
  \item[1.4] Type-1 evolutions $y\mapsto\beta^y$ are path-continuous Hunt processes: they can be started from any interval partition in a Lusin state space $(\mathcal{I}_H,d_H)$, are
    continuous in the initial condition and satisfy the strong Markov property.
  \item[3.2] The level-$y$ aggregate mass process $s\mapsto M_{\bN,\bX}^y(s)$ of \eqref{eq:agg_mass_from_spindles}, time-changed by the inverse local time $\tau^y_\bX$ of $\bX$ at level $y$ is a stable
    subordinator of index $\alpha$. 
  \item[4.9] $\bF^y$ is a \PRM, whose intensity measure we call \em bi-clade excursion measure\em.
  \item[4.11] Bi-clades are space/time-reversible in the sense that reversing scaffolding time and block diffusion time in spindles yields
    the same bi-clade excursion measure.
  \item[4.15] Mid-bi-clade Markov property: conditionally given a spindle mass $f_R(\mathbf{X}_{R-})=x$ of 
    the spindle $(R,f_R)$ in $\mathbf{N}$ at the time $R$ when the scaffolding $\mathbf{X}$ crosses the line $\mathbf{X}=0$, the clade part (post-$R$) and the time-reversed anti-clade part (pre-$R$) are independent and 
    distributed as clades starting from $x$.
  \item[5.5] The skewer processes of $(\mathbf{N},\mathbf{X})$ stopped at stopping times including $\tau_\mathbf{X}^0(u)$ and $S_\mathbf{X}(-u)$ for $u\ge 0$
    are type-1 evolutions. 
\end{enumerate}

\noindent In \cite{IPPB}, we further prove the following.
\begin{enumerate}
  \item[1.2] The type-1 semi-group $\kappa_y^\mathcal{I}$, $y\!\ge\!0$, is as follows. Independently for each block $V\!\in\!\beta^0$ of size $b\!=\!{\tt Leb}(V)$, there is a contribution to 
     level $y$ with probability $1\!-\!e^{-b/2y}$. 
     Such a contribution consists of a left-most interval with Laplace transform $(1\!+\!\gamma/r)^\alpha(e^{br^2/(r+\gamma)}\!-\!1)/(e^{br}\!-\!1)$, 
     where $r\!=\!1/2y$, 
     concatenated with a scaled ${\tt PDIP}(\alpha,\alpha)$, the interval partition formed by the excursion intervals of a ${\tt BES}(2-2\alpha)$-bridge,
     scaled by an independent ${\tt Gamma}(\alpha,1/2y)$-distributed random factor. The contributions are
     concatenated in the order of $V\in\beta^0$ to give the distribution under the time-$y$ transition kernel starting from $\beta^0$.
  \item[1.3] The kernels $\widetilde{\kappa}_y^\mathcal{I}$ obtained by concatenating 
     a ${\tt PDIP}(\alpha,\alpha)$  scaled by an independent ${\tt Gamma}(\alpha,1/2y)$ with the interval partition from $\kappa_y^{\mathcal{I}}$, $y\!\ge\!0$, 
     also form the semi-group of a path-continuous Hunt process, called \em type-0 evolution\em.
  \item[1.4]\begin{enumerate}
                    \item[(i)] The total mass process $(\IPmag{\beta^y},y\!\ge\! 0)$ of a type-1 evolution is $\BESQ_{\IPmag{\beta^0}}(0)$, 
                      for all initial $\beta^0\in\mathcal{I}_H$, hence including the case of a single clade.
                   \item[(ii)] The total mass processes of type-0 evolutions are $\BESQ_{\IPmag{\beta^0}}(2\alpha)$.
                 \end{enumerate}
  \item[3.2] \begin{enumerate}\item[(i)] The \PRM\ $\mathbf{F}^y$ has points at bi-clades whose value of the central spindle mass when crossing $\mathbf{X}=0$
                                                                  has (sigma-finite) law $(\alpha/\Gamma(1-\alpha)) x^{-\alpha-1}dx$, and
                                                  \item[(ii)] points at bi-clades with $\mathbf{X}$-supremum above $y$ at rate $2^{-\alpha}y^{-\alpha}$.
                   \end{enumerate}
  \item[3.10] $(\skewer(y,\mathbf{N}|_{[0,S_\mathbf{X}(-u)]\times\Exc},u\!+\!\mathbf{X}|_{[0,S_\mathbf{X}(-u)]}),0\!\le\!y\!\le\! u)$ is a type-0 evolution. 
\end{enumerate}

\section{Construction of $(\mathbf{X},\mathbf{N})$ from $(\mathbf{R},\mathbf{H})$ and vice versa}\label{sec:new}

\noindent Let $\tau_\mathbf{R}^0(s)\!=\!\inf\{L^0(t)\!>\!s\}$, $s\!\ge\! 0$, be the inverse local time of $\mathbf{R}$ at $0$, 
and $\mathbf{K}$ the \PRM\ of excursions of $\mathbf{R}$ away from $0$. For each excursion interval
$(\tau^0_\mathbf{R}(s-),\tau^0_\mathbf{R}(s))=(\ell,r)$ of $\mathbf{R}$, we decompose the Bessel excursion $(\mathbf{R}(\ell+t),0\le t\le r-\ell)$ in $\mathbf{K}$ as in 
$\mathbf{R}=\mathbf{B}-(1-d)\mathbf{H}$ in \eqref{besdec}, and we define the associated occupation density local time process $\lambda_s:=(\lambda^y_s,y\ge 0)$ of 
$(\mathbf{H}(\ell+t)-\mathbf{H}(\ell),0\le t\le r-\ell)$. 

\begin{proposition}\label{prop:Bertoin} The random measure $\sum_{s\ge 0\colon\tau^0_\mathbf{R}(s-)<\tau^0_\mathbf{R}(s)}\delta_{\lambda_s}$ is 
  a $\PRM\big({\tt Leb}\otimes\overline{\nu}\big)$, where $\overline{\nu}$ is a Pitman--Yor excursion measure associated with \BESQ[-2(1-d)], the process $\mathbf{H}\circ\tau^0_\mathbf{R}$ is a spectrally 
  positive stable process of index $2\!-\!d$. The pair has the same distribution as $(\mathbf{N},\mathbf{X})$ in Section \ref{sec:prel}, with $\alpha=1-d$,
  up to a linear time-change.
\end{proposition}
\begin{proof} Since $t\mapsto\mathbf{H}(t)$ is differentiable almost everywhere, with derivative $\mathbf{H}^\prime(t)=1/2\mathbf{R}(t)$, 
  its local time at level $\mathbf{H}(t)$ increases by a jump of $2\mathbf{R}(t)$ at time $t$. Specifically, during each excursion interval $(\tau^0_\mathbf{R}(s-),\tau^0_\mathbf{R}(s))=(\ell,r)$ of $\mathbf{R}$ away from 0, 
  we get 
  \begin{equation}\label{eq:excloctime}\lambda_s^{\mathbf{H}(\ell+t)-\mathbf{H}(\ell)}=2\mathbf{R}(\ell+t),\quad 0\le t\le r-\ell,
  \end{equation}
  i.e.\ the local times of $\mathbf{H}$ during excursions are continuous time-changes of the excursions of $2\mathbf{R}$. 
  Hence, the jump sizes of $\mathbf{H}\circ\tau^0_\mathbf{R}$ are 
  $$\mathbf{H}(\tau^0_\mathbf{R}(s))-\mathbf{H}(\tau^0_\mathbf{R}(s-))=\inf\{y>0\colon\lambda^y_s=0\}=\sup\{y\ge 0\colon\lambda^y_s>0\},\ s\ge 0.$$ 
  Bertoin \cite[Proof of Lemma 3.2]{Bertoin1990} showed that $\mathbf{H}\circ\tau^0_\mathbf{R}$ is a spectrally positive stable process of index 
  $2\!-\!d$. Furthermore, by standard mapping of \PRM s, $\sum_{s\ge 0\colon\tau^0_\mathbf{R}(s-)<\tau^0_\mathbf{R}(s)}\!\delta_{\lambda_s}$ is a \PRM. 
  We will identify its intensity measure as a \BESQ[-2(1-d)] excursion measure by \cite[(3.1) First description]{PitmYor82}. 
  
  Specifically, this description requires us to check three points. 
  (i) Neither excursion measure charges the zero excursion. 
  (ii) Whether the hitting time $T_x$ of level $x$ by the excursion of $\mathbf{R}$ is finite or not is not affected by the time-change,
  and the associated rate under either excursion measure is proportional to $1/s(x)$ 
  where $s(x)=x^{2-d}$ is the common scale function of ${\tt BES}(d)$ and \BESQ[-2(1-d)], see e.g. \cite[(3.5) Examples]{PitmYor82}.   
  (iii) We will show that the pre-$T_x$ and post-$T_x$ processes are, as required. 
  The pre-$T_x$ part of the excursion of $\mathbf{R}$ is a ${\tt BES}_0(d)$ conditioned to stay positive, 
  i.e. a ${\tt BES}_0(4-d)$ (see again \cite[(3.5) Examples]{PitmYor82}). 
  The time-change relation \eqref{eq:excloctime} transforms this into a $\BESQ_0(4+2(1-d))$, by \cite[Proposition XI.(1.11)]{RevuzYor}, 
  which is a $\BESQ_0(-2(1-d))$ conditioned to stay positive, as required. 
  Bertoin \cite[bottom of p.\ 117]{Bertoin1990c} noted the corresponding time-change relation 
  for ${\tt BES}(d)$ and \BESQ[-2(1-d)] starting from $y$ stopped when hitting 0. This identifies the post-$T_x$ parts of the
  excursions and completes the proof.
\end{proof}

\begin{proof}[Proof of Theorems \ref{thm:Bertoin} and \ref{thm:diffusion_0}] This will follow from Proposition \ref{prop:Bertoin} because applying $\skewer(y,\cdot)$ to the scaffolding-and-spindles pair of the proposition yields an interval partition with blocks 
\begin{equation}\label{corrr}
  \lambda_s^{y-\mathbf{H}(\tau^0_\mathbf{R}(s-))}=2\mathbf{R}(t)=\lambda^y(t)\!-\!\lambda^y(t-)\quad\mbox{if }y\!=\!\mathbf{H}(t)\mbox{ and }t\!\in\!(\tau^0_\mathbf{R}(s-),\tau^0_\mathbf{R}(s)),
\end{equation}
  by \eqref{eq:excloctime} and \eqref{eq:loctime}. Since $\skewerP(y,\mathbf{N},\mathbf{X})$ is unaffected by (linear) changes of scaffolding time of
  $\mathbf{X}$ and $\mathbf{N}$, the process of Theorem \ref{thm:Bertoin} can be constructed as claimed in Theorem \ref{thm:diffusion_0}, when stopped at a time
  that corresponds to an inverse local time $\tau_{\mathbf{R},\mathbf{H}}^{(0,0)}(u)$ of $(\mathbf{R},\mathbf{H})$ at $(0,0)$ and that after time change by $\tau^0_\mathbf{R}$, is an inverse local time 
  of the ${\tt Stable}(2-d)$ process $\mathbf{H}\circ\tau^0_\mathbf{R}$. 
    
Let us work out the constant $c$ for which stopping $(\mathbf{X},\mathbf{N})$ at $\tau_{\mathbf{X}}^0(cu)$ yields the same initial distribution for the skewer process
as the stopped scaffolding-and-spindles pair constructed from $(\mathbf{R},\mathbf{H})$ stopped at $\tau_{\mathbf{R},\mathbf{H}}^{(0,0)}(u)$. 
We do this using the parts of \cite[Lemma 3.3]{Bertoin1990} and \cite[Proposition 3.2]{IPPB} that we recalled in Sections \ref{sec:Bertoin} and \ref{sec:prel} here. 
Specifically, the statistics of excursions of $(\mathbf{R},\mathbf{H})$ of $\mathbf{H}$-infima 
directly transfer to $\mathbf{H}\circ\tau_{\mathbf{R}}^0$-infima that correspond to $\mathbf{X}$-infima, which, by bi-clade reversibility (see 4.11 above) or the 
mid-bi-clade Markov property (see 4.15 above) have the same rates as $\mathbf{X}$-suprema in a bi-clade. 
But the rates of $\mathbf{H}$-infima and $\mathbf{X}$-suprema differ by the constant $c=2^\alpha=2^{1-d}$, 
hence $\mathbf{X}$ needs to run longer than $\mathbf{H}\circ\tau_{\mathbf{R}}^0$, by a factor of $c$, 
to achieve the same number of excursions exceeding any given level $y$.

  Finally, we note that the skewer process associated with $(\mathbf{N}|_{[0,\tau_{\mathbf{X}}^0(v)]\times\Exc},\mathbf{X}|_{[0,\tau_{\mathbf{X}}^0(v)]})$ is a diffusion by \cite[Theorem 1.4]{IPPA}, again as recalled in Section \ref{sec:prel} here. 
\end{proof}

A similar argument to work out $c$ can be based on the values of $\mathbf{R}$ when crossing $\mathbf{H}=0$ and the mass of the central spindle of $\mathbf{N}$ when crossing $\mathbf{X}=0$. Note, however, that these also differ by a factor of 2, by \eqref{eq:excloctime}. 
  
\begin{corollary}\label{cor:spindles} In Bertoin's setting, under the It\^o excursion measure of $\mathbf{R}$, the local time process of $\mathbf{H}$ has as its law a 
  Pitman--Yor excursion measure of \BESQ[-\!2(1\!-\!d)].
\end{corollary}

Proposition \ref{prop:Bertoin} makes precise the sense 
in which the framework of a single Bessel process $\mathbf{R}\sim{\tt BES}_0(d)$ of \cite{Bertoin1990,Bertoin1990c}, 
via $(\mathbf{R},\mathbf{H})$, 
yields the scaffolding-and-spindles framework $(\mathbf{N},\mathbf{X})$ of \cite{IPPA,IPPB}. 
The main step in the proof is time-changing the excursions of $\mathbf{R}$ away from 0 to form ${\tt BESQ}(-2(1-d))$ spindles. 
Let us invert this time-change and construct $\mathbf{R}$ from the spindles of $\mathbf{N}$. 
To this end, recall our notation $\nu$ for the Pitman--Yor excursion measure of ${\tt BESQ}(-2\alpha)$ of Section \ref{sec:prel}. 

\begin{proposition}\label{prop:reverse} For $\mathbf{N}=\sum_{i\in\mathbf{I}}\delta_{(\mathbf{s}_i,\mathbf{f}_i)}\sim{\tt PRM}({\tt Leb}\otimes\nu)$, 
  set $\zeta_i:=\int_0^{\zeta(\mathbf{f}_i)}\mathbf{f}_i(y)dy$ and
  $$\mathbf{e}_i(t):=\frac{1}{2}\mathbf{f}_i\left(\inf\left\{z\ge 0\colon\int_0^z\mathbf{f}_i(y)dy>t\right\}\right),\quad 
      t\in\left[0,\zeta_i\right),\quad\mbox{and}\quad\mathbf{e}_i(\zeta_i)=0.
  $$
  Then $\sum_{i\in\mathbf{I}}\delta_{(\mathbf{s}_i,\mathbf{e}_i)}$ has the same distribution as the It\^o excursion process $\mathbf{K}$ of 
  $\mathbf{R}\sim{\tt BES}_0(d)$, up to a linear time change, with $d=1-\alpha$. 
  In particular, the $\mathbf{e}_i$ can be stitched together in the order of the $\mathbf{s}_i$, $i\in\mathbf{I}$, to yield a process $\overline{\mathbf{R}}\sim{\tt BES}_0(1-\alpha)$.   
\end{proposition}
\begin{proof} This follows from Proposition \ref{prop:Bertoin}. Specifically, mapping $\mathbf{f}_i$ to $\mathbf{e}_i$ is elementary 
  since all $\mathbf{f}_i$ are continuous with compact support a.s.. In present notation, we can write \eqref{eq:excloctime} as
  $$\mathbf{f}_i\left(\int_0^t\frac{du}{\mathbf{e}_i(u)}\right)=2\mathbf{e}_i(t),\quad t\in[0,\zeta_i].$$
  This is a.s.\ well-defined for all $\mathbf{e}_i$, $i\in\mathbf{I}$, so the time-changes relating $\mathbf{f}_i$ and $\mathbf{e}_i$ are bijective, and hence the 
  associated \PRM s are bijectively related by standard mapping of \PRM s. In particular, we deduce the claimed distributional identities up to a linear time change.
  The construction of Markov processes from excursions has been well-studied \cite{Salisbury1986a}. 
  Note that a linear time change of the \PRM\ has no effect on the ${\tt BES}_0(1-\alpha)$-excursions themselves. Specifically, we define
  $\overline{\tau}(s)=\sum_{i\in\mathbf{I}\colon\mathbf{s}_i\le s}\zeta_i$, $s\ge 0$, and $\overline{\mathbf{R}}(\overline{\tau}(\mathbf{s}_i-)+t)=\mathbf{e}_i(t)$, 
  $0\le t\le\zeta_i$, also setting $\overline{\mathbf{R}}(t)=0$ for $t\not\in\bigcup_{i\in\mathbf{I}}[\overline{\tau}(\mathbf{s}_i-),\overline{\tau}(\mathbf{s}_i)]$ and
  obtain $\overline{\mathbf{R}}\sim{\tt BES}_0(1\!-\!\alpha)$, and this is the same process as if we replace $\mathbf{s}_i$ by $a\mathbf{s}_i$, $i\in\mathbf{I}$, 
  throughout, $a\!>\!0$. The process $\overline{\tau}$ is an inverse local time of $\overline{\mathbf{R}}$ at 0, and replacing $\mathbf{s}_i$ by $a\mathbf{s}_i$ corresponds to
  a different choice of local time.
\end{proof}

\begin{corollary} For $(\mathbf{X},\mathbf{N})$ as in Section \ref{sec:prel} and notation $\overline{\mathbf{R}}$ as in Proposition \ref{prop:reverse}, with 
  $\overline{\tau}(s)=\sum_{i\in\mathbf{I}\colon\mathbf{s}_i\le s}\zeta_i$, $s\ge 0$, define $\overline{\mathbf{H}}$ on the range of $\overline{\tau}$ as 
  $\overline{\mathbf{H}}(\overline{\tau}(s)):=\mathbf{X}(s)$, $s\ge 0$, and outside the range of $\overline{\tau}$ as 
  $$\overline{\mathbf{H}}\left(\overline{\tau}(\mathbf{s}_i-)+\frac{1}{2}\int_0^z\mathbf{f}_i(y)dy\right):=\mathbf{X}(\mathbf{s}_i-)+z,\quad 0\le z<\Delta\mathbf{X}(\mathbf{s}_i)=\zeta(\mathbf{f}_i).$$
  Then the pair $(\overline{\mathbf{R}},\overline{\mathbf{H}})$ has the same distribution as $(\mathbf{R},\mathbf{H})$ of Section \ref{sec:Bertoin}.
\end{corollary}
\begin{proof} Since $\mathbf{H}$ is determined by $\mathbf{R}$ via \eqref{eq:defH} and $\overline{\mathbf{R}}\stackrel{d}=\mathbf{R}$, it suffices to show that $\overline{\mathbf{H}}$ relates to 
  $\overline{\mathbf{R}}$ in the same way. Indeed, we have $\overline{\mathbf{H}}\circ\overline{\tau}=\mathbf{X}$, by construction, and $\mathbf{X}(s)$ is the
  compensated limit of its jumps $\Delta\mathbf{X}(\mathbf{s}_i)\!=\!\zeta(\mathbf{f}_i)$ for $i\!\in\!\mathbf{I}$ with $\mathbf{s}_i\!\le\! s$. But 
  $$\zeta(\mathbf{f}_i)=\int_0^{\zeta(\mathbf{e}_i)}\frac{du}{\mathbf{e}_i(u)}=\int_0^\infty a^{d-2}L^a_i(\infty)da,$$
  where $(a^{d-1}L^a_i(\infty),a\ge 0)$ is the continuous version of the total occupation density local time of $\mathbf{e}_i$ at level $a$. This entails that the right-most equality of 
  \eqref{eq:defH} holds for $t=\overline{\tau}(s)$, when $(\mathbf{R},\mathbf{H})$ is replaced by $(\overline{\mathbf{R}},\overline{\mathbf{H}})$. Since these 
  limits exist almost surely uniformly for $s$ in compact intervals, they also hold at $t=\overline{\tau}(\mathbf{s}_i-)$, $i\in\mathbf{I}$.
  Beyond the range of $\overline{\tau}$, we have, for each $i\in\mathbf{I}$,
  $$\overline{\mathbf{H}}\left(\overline{\tau}(\mathbf{s}_i-)+\frac{1}{2}\int_0^z\mathbf{f}_i(y)dy\right)=\overline{\mathbf{H}}(\overline{\tau}(\mathbf{s}_i-))+z,
      \quad 0\le z\le\zeta(\mathbf{f}_i).$$
  But according to the bijective time change relationships transforming $\mathbf{f}_i$ into $\mathbf{e}_i$ noted in the proof of Proposition \ref{prop:reverse}, we have
  $$z=\int_0^t\frac{du}{\mathbf{e}_i(u)}\quad\mbox{if and only if}\quad t=\frac{1}{2}\int_0^z\mathbf{f}_i(y)dy.$$
  Hence, we obtain  
  $$\overline{\mathbf{H}}(\overline{\tau}(\mathbf{s}_i-)+t)=\overline{\mathbf{H}}(\overline{\tau}(\mathbf{s}_i-))+\int_0^t\frac{du}{\mathbf{e}_i(u)},
      \quad 0\le t\le\zeta(\mathbf{e}_i),$$
  and this completes the proof.
\end{proof}

\section{Further consequences of the connection between \cite{Bertoin1990,Bertoin1990c} and \cite{IPPA,IPPB}}\label{sec:new2}

\begin{table}[b]
  \begin{center}
   \begin{tabular}{|c||c|c|c|c|c|c||c|c|c|c|}
    \hline
     &&&&&&&&&&\\[-0.3cm]
    \cite{IPPA}$\|$\cite{IPPB}                          &1.3 & 1.4 & 4.9 & 4.11 & 4.15&5.5& 1.2 &$\!$1.2, 1.3, 3.10$\!$&    1.4            &3.2\\[0.1cm]
    \hline
    &&&&&&&&&&\\[-0.3cm]
    \cite{Bertoin1990,Bertoin1990c} &II.1& II.2 &3.1 & 3.2 & 3.4 &   II.2& II.3&     II.2            &$\!$4.1, 4.2, II.4$\!$&3.3\\[0.1cm]
    \hline
  \end{tabular}
  \bigskip
  
  \caption{Each column lists pairs of results (or groups of results) from \cite{IPPA} or \cite{IPPB}, and from \cite{Bertoin1990,Bertoin1990c} that are analogues of each other.\label{table}}
  \end{center}
\end{table}

In the light of the results of Section \ref{sec:new}, the results of \cite{Bertoin1990,Bertoin1990c} and \cite{IPPA,IPPB} are closely related. Indeed, many results of \cite{Bertoin1990,Bertoin1990c} can now be deduced from \cite{IPPA,IPPB}, and the approach of \cite{Bertoin1990,Bertoin1990c} could be refined to handle the
additional order structure needed for the interval partitions of \cite{IPPA,IPPB}. Table \ref{table} pairs the analogous results, which will mostly have 
been evident already from the formulations in Sections \ref{sec:Bertoin} and \ref{sec:prel}.

One may note, however, that these results differ in detail, not just because $(\beta^y,y\!\ge\!0)$ and $(\mu^y_{[0,T]},y\!\ge\! 0)$ have different state spaces. Specifically,
\cite[Theorem 1.4]{IPPA} and \cite[Theorem 1.3]{IPPB} establish interval-partition-valued processes as path-continuous Hunt processes that are continuous in the
initial condition, while \cite[Theorem II.2]{Bertoin1990c} does not push beyond the simple Markov property. On the other hand,  
\cite[Proposition 2.4]{Bertoin1990} and \cite[Proposition 3.2]{IPPA} find stable inverse local times of different indices, but fundamentally play the same role, since they
provide the time parameterisations for the \PRM s of excursions of $(\mathbf{R},\mathbf{H})$ and of bi-clades, respectively. 

The observation of Corollary \ref{cor:spindles}, that $\mathbf{H}$-local time processes in $\mathbf{R}$-excursions are $\BESQ(-2(1-d))$, is related to 
\cite[Theorem I.5 or (0.3)]{Bertoin1990c}, which notes $\BESQ(-2(1-d))$ evolution of time-changed $\mathbf{R}$-excursions after they exceed previous 
$\mathbf{H}$-suprema. In the context of \cite{IPPA,IPPB}, the corresponding result is a consequence of the construction from $\BESQ(-2\alpha)$ spindles 
(and the Markov property). But \cite[Theorem I.5]{Bertoin1990c} goes further and yields the following result when translated into the framework of \cite{IPPA,IPPB}.

\begin{corollary} For each $y\ge 0$, let $T_\mathbf{X}^+(y)=\inf\{s\ge 0\colon\mathbf{X}(s)>y\}$ and denote by 
  $\mathbf{L}(y):=\mathbf{f}_{T_\mathbf{X}^+(y)}(y-\mathbf{X}(T_\mathbf{X}^+(y)-))$ the value of the left-most spindle $\mathbf{f}_{T_\mathbf{X}^+(y)}$ 
  that crosses level $y$. Then $(\mathbf{L}(y),y\!\ge\! 0)$ is a Markov process whose excursions away from 0 start with a jump of intensity 
  $(2^{\alpha}\alpha/\Gamma(1-\alpha))x^{-1-\alpha}dx$ and then evolve as $\BESQ_x(-2\alpha)$.
\end{corollary}

Similarly, \cite[Theorem I.6]{Bertoin1990c} then yields the semi-group of $\mathbf{L}$.

Less immediate are the consequences of some further results of \cite{IPPB}, which we have not stated in Section \ref{sec:prel}, about what we call pseudo-stationarity
of type-0 and type-1 evolutions, and the passage to normalised interval-partition evolutions on the subspace $\mathcal{I}_{H,1}$ of interval partitions via suitable 
time-change. 
We observe in the context of Bertoin \cite[Theorem II.2]{Bertoin1990c} that the marginal distributions of 
$\mu_{[0,T_\mathbf{H}(-1)]}^{-1+y}$, $y\in[0,1]$, a process starting from the zero measure $0\in\mathcal{N}((0,\infty))$, can be read from
\begin{equation}\label{pseudostat1}
  \widetilde{\kappa}_y^\mathcal{N}f_\varphi(0)=\frac{y^{d-1}}{y^{d-1}\!+\!\int_0^\infty(1\!-\!\varphi(s))\Pi_y(ds)}=\mathbb{E}\Bigg(\!f_\varphi\Bigg(\sum_{j\in\mathbf{J}\colon 0\le\mathbf{r}_j\le E_y}\!\delta_{\Delta\sigma_y(\mathbf{r}_j)}\!\Bigg)\!\Bigg),
\end{equation}
where $\Pi_y(ds)\!=\!((1\!-\!d)/\Gamma(d))s^{d-2}e^{-s/y}ds$ is the L\'evy measure of a subordinator $(\sigma_y(r),r\!\ge\! 0)$ with \PRM\  
$\sum_{j\in\mathbf{J}}\delta_{(\mathbf{r}_j,\Delta\sigma_y(\mathbf{r}_j))}$ of its jumps, and $E_y\sim{\tt Exp}(y^{d-1})$ is an independent random variable. Now \cite[Proposition 21]{PitmYorPDAT} showed, for $y=1$, that the decreasing rearrangement  
$(\sigma_y(E_y))^{-1}(\Delta\sigma_y(t),0\le t\le E_y)^\downarrow$ of normalised jump sizes of $\sigma_y|_{[0,E_y]}$ has Poisson--Dirichlet 
distribution ${\tt PD}(1-d,1-d)$, and is independent of $\sigma_y(E_y)\sim{\tt Gamma}(1-d,1/y)$, and a simple change of variables extends this to all $y>0$. As ${\tt PD}(1-d,1-d)$ is preserved at all times $y>0$ (while ${\tt Gamma}(1-d,1/y)$ depends on $y>0$), we call this behaviour
\em pseudo-stationarity\em, cf. \cite[Theorem 1.5]{IPPB}. 

Furthermore, it is well-known that adding an independent $\mathbf{R}(0)\sim{\tt Gamma}(d,1/y)$ variable to the jumps of 
$\sigma_y|_{[0,E_y]}$, we obtain $\mathbf{R}(0)+\sigma_y(E_y)\sim{\tt Exp}(1/y)$ independent of
$(\mathbf{R}(0)+\sigma_y(E_y))^{-1}(\mathbf{R}(0);\Delta\sigma_y(t),0\le t\le E_y)^\downarrow\sim{\tt PD}(1-d,0)$, and there is 
pseudo-stationarity in the following sense.

\begin{theorem}\label{pseudo} Let $\mathbf{R}$ be a ${\tt BES}(d)$ starting from $\mathbf{R}(0)\sim{\tt Gamma}(d,\rho)$ and let
  $\mathbf{Q}$ be a ${\tt BESQ}(0)$ starting from $\mathbf{Q}(0)\sim{\tt Exp}(\rho/2)$ independent of 
  $(\mathbf{x}_i,i\ge 1)\sim{\tt PD}(1-d,0)$. Then $\mu_{[0,T_{\mathbf{H}}(-1/\rho)]}^y$ has the same distribution as
  $\sum_{i\ge 1}\delta_{\mathbf{x}_i\mathbf{Q}(y)/2}$, for each fixed $y>0$.  
\end{theorem}
\begin{proof} Let $\mathbf{f}_0:=\lambda_0=(\lambda_0^y,y\ge0)$ be the local time process of $(\mathbf{H}(t),0\le t\le T_{\mathbf{R}}(0))$. 
  Then $\mathbf{f}_0(y)=\widetilde{\mathbf{R}}(y)$ for all $0\le y<\mathbf{H}(T_\mathbf{R}(0))=\zeta(\mathbf{f}_0)$, 
  and by \cite[Theorem I.5]{Bertoin1990c}, $\mathbf{f}_0$ is a ${\tt BESQ}(-2(1-d))$,  
  starting from $2\mathbf{R}(0)\sim{\tt Gamma}(d,\rho/2)$. 
  Proceeding as in Proposition \ref{prop:Bertoin}, where $\mathbf{R}(0)=0$, we obtain here, 
  after the linear time-change noted in that proposition, a point measure $\mathbf{N}_0:=\delta_{(0,\mathbf{f}_0)}+\mathbf{N}$, in which 
  $\mathbf{N}\sim{\tt PRM}({\tt Leb}\otimes\nu)$ is independent of $\mathbf{f}_0\sim{\tt BESQ}_{2\mathbf{R}(0)}(-2(1-d))$ and scaffolding 
  $\mathbf{X}_0:=\zeta(\mathbf{f}_0)+\mathbf{X}$.
  
  By the strong Markov property of $(\mathbf{R},\mathbf{H})$ at $T_\mathbf{H}(0)$, and by \eqref{pseudostat1} with $y=1/\rho$, we find that 
  $\mu_{[0,T_{\mathbf{H}}(-1/\rho)]}^0$ has the claimed initial distribution.   
  Similarly, but now based on the strong Markov property of $(\mathbf{N}_0,\mathbf{X}_0)$ at $S_{\mathbf{X}_0}(0)$ 
  and on \cite[Theorem 1.3]{IPPB}, $(\mathbf{N}_0,\mathbf{X}_0)$ correspondingly 
  stopped at $S_{\mathbf{X}_0}(-1/\rho)$, has a skewer process $(\beta^y,y\!\ge\! 0)$ starting from a ${\tt PDIP}(\alpha,0)$ 
  scaled by the independent ${\tt Exp}(\rho/2)$. By \cite[Theorem 1.5]{IPPB}, $(\beta^y,y\!\ge\!0)$ is pseudo-stationary
  with $\beta^y$ distributed as a ${\tt PDIP}(\alpha,0)$ scaled by an independent $\mathbf{Q}(y)$, where $\mathbf{Q}$ is a ${\tt BESQ}(0)$
  starting from $\mathbf{Q}(0)\sim{\tt Exp}(\rho/2)$.
  
  But as $(\mathbf{N}_0,\mathbf{X}_0)$ has been constructed from $(\mathbf{R},\mathbf{H})$ as Proposition \ref{prop:Bertoin} did for the proof of Theorems \ref{thm:Bertoin} 
  and \ref{thm:diffusion_0}, we read from \eqref{corrr} the coupling
  \begin{equation}\label{coupling}\mu_{[0,T_\mathbf{H}(-1/\rho)]}^y=\phi(\beta^y),\qquad\mbox{where }\phi(\beta)=\sum_{V\in\beta}\delta_{{\tt Leb}(V)/2},
  \end{equation}
  so the distribution of $\mu_{[0,T_\mathbf{H}(-1/\rho)]}^y$ follows from the distribution of the ranked sequence of interval lengths of the pseudo-stationary $\beta^y$, which are
  ${\tt PD}(\alpha,0)$ scaled by independent $\mathbf{Q}(y)$, as required.
\end{proof}

In the light of this coupling \eqref{coupling}, \cite[Theorem 1.6]{IPPB} has the following corollary. Let
$\mathcal{N}_1((0,\infty)):=\big\{\sum_{i\in I}\delta_{x_i}\in\mathcal{N}((0,\infty))\colon\sum_{i\in I}x_i=1\big\}$ and consider the map
$\mu=\sum_{i\in I}\delta_{x_i}\mapsto\overline{\mu}:=\sum_{i\in I}\delta_{x_i/\sum_{j\in I}x_j}$ from $\mathcal{N}((0,\infty))\setminus\{0\}$ to $\mathcal{N}_1((0,\infty))$. 

\begin{corollary} Let $\mathbf{R}$ be as in Theorem \ref{pseudo}, set $T:=T_\mathbf{H}(-1)$ and consider the time-change
  $$\varrho(u):=\inf\left\{y\ge 0\colon\int_0^y\frac{dz}{\lambda^z(T)}>u\right\},\qquad u\ge 0.$$
  Then the process $\Big(\overline{\mu}_{[0,T]}^{\varrho(u)}, u\!\ge\! 0\Big)$, obtained from $\big(\mu_{[0,T]}^y,y\!\ge\!0\big)$ by
  first time-changing by $\varrho$ and then mapping under $\mu\mapsto\overline{\mu}$, is a stationary Markov process whose invariant distribution 
  is the distribution of $\sum_{i\ge 1}\delta_{\mathbf{x}_i}$ for $\big(\mathbf{x}_i,i\ge 1\big)\sim{\tt PD}(\alpha,0)$.
\end{corollary}
\begin{proof} By \cite[Theorem 1.6]{IPPB}, the corresponding interval-partition-valued process is a stationary Markov process, whose invariant
  distribution is ${\tt PDIP}(\alpha,0)$. Specifically, for $\beta\!\in\!\mathcal{I}_H\!\setminus\!\{\emptyset\}$, let 
  $\overline{\beta}:=\{(a/\|\beta\|,b/\|\beta\|)\colon(a,b)\!\in\!\beta\}
    \in\mathcal{I}_{H,1}:=\{\gamma\!\in\!\mathcal{I}_H\colon\|\gamma\|\!=\!1\}$. 
  Recall that with the coupling \eqref{coupling}, for $\rho\!=\!1$, we have $\lambda^y(T)=\|\beta^y\|$ for all $y\ge 0$. Then 
  $$\overline{\mu}^{\varrho(u)}_{[0,T]}=\sum_{V\in\beta^{\varrho(u)}}\delta_{{\tt Leb}(V)/\|\beta^{\varrho(u)}\|}
                                                         =\sum_{V\in\overline{\beta}^u}\delta_{{\tt Leb}(V)}.$$
  This yields that each $\overline{\mu}^{\varrho(u)}_{[0,T]}$ has the claimed distribution. The Markov property will follow from Dynkin's 
  criterion for when a function of a Markov process is a Markov process. Specifically, any two $\beta,\gamma\!\in\!\mathcal{I}_H$ with 
  $\phi(\beta)\!=\!\phi(\gamma)$ have intervals of the same lengths, so 
  there is a bijection $\eta\colon\gamma\rightarrow\beta$ such that ${\tt Leb}(\eta(V))={\tt Leb}(V)$ for all $V\!\in\!\gamma$. We construct
  coupled type-1 evolutions $(\beta^y,y\!\ge\!0)$ starting from $\beta$ and $(\gamma^y,y\!\ge\!0)$ starting from $\gamma$ by building 
  $(\mathbf{X}_\beta,\mathbf{N}_\beta)$ from $(\mathbf{X}_V,\mathbf{N}_V)$, $V\!\in\!\beta$, as in Section \ref{sec:prel}, and then building 
  $(\mathbf{X}_\gamma,\mathbf{N}_\gamma)$ by stitching together the same $(\mathbf{X}_{\eta(V)},\mathbf{N}_{\eta(V)})$, $V\!\in\!\gamma$, 
  in the order given by $\gamma$. Then $\phi(\beta^y)\!=\!\phi(\gamma^y)$ and $\phi(\overline{\beta}^u)\!=\!\phi(\overline{\gamma}^u)$. 
  In particular, the distributions of $\phi(\overline{\beta}^u)$ and $\phi(\overline{\gamma}^u)$ coincide, as required for Dynkin's criterion.
\end{proof}


Finally, we rewrite $\mathbf{R}(t)=\mathbf{B}(t)-(1\!-\!d) \mathbf{H}(t)$ as a decomposition of Brownian motion 
$\mathbf{B}(t)=\mathbf{R}(t)+(1-d)\mathbf{H}(t)$. In Proposition \ref{prop:Bertoin}, we time-changed $\mathbf{H}$ by the inverse local time $\tau_\mathbf{R}^0$. But then $(1\!-\!d)\mathbf{H}(\tau^0_\mathbf{R}(s))=\mathbf{B}(\tau_\mathbf{R}^0(s))$ is a time-changed Brownian motion. Also,
during each jump $\mathbf{H}(\tau^0_\mathbf{R}(s))-\mathbf{H}(\tau_\mathbf{R}^0(s-))$, we can write
$$\;\mathbf{B}(\tau^0_\mathbf{R}(s-)\!+\!t)-\mathbf{B}(\tau^0_\mathbf{R}(s-))
							=\mathbf{R}(\tau^0_\mathbf{R}(s-)\!+\!t)+(1\!-\!d)\!\left(\mathbf{H}(\tau_\mathbf{R}^0(s-)\!+\!t)\!-\!\mathbf{H}(\tau_\mathbf{R}^0(s-))\right)\!,\;$$
which is the part of the Brownian motion from which the corresponding excursion of $\mathbf{R}$ away from 0 is built. From that excursion, we built
the corresponding (increasing) stretch of $(1-d)\mathbf{H}$, whose local time is a ${\tt BESQ}(-2(1-d))$ excursion of length
$\mathbf{H}(\tau_\mathbf{R}^0(s))-\mathbf{H}(\tau_\mathbf{R}^0(s-))$, by Corollary \ref{cor:spindles}. Note that this part of the Brownian motion is
positive (relative to its starting level) and indeed it stays above the increasing stretch of $(1-d)\mathbf{H}$ since $\mathbf{R}(t)>0$
during $(\tau_\mathbf{R}^0(s-),\tau_\mathbf{R}^0(s))$. 

At $\tau_{\mathbf{R},\mathbf{H}}^{(0,0)}(u)$, when $\mathbf{R}$ and $\mathbf{H}$ both vanish, we also have 
$\mathbf{B}(\tau_{\mathbf{R},\mathbf{H}}^{(0,0)}(u))=0$. Bertoin \cite[Lemma 3.2]{Bertoin1990} noted that 
$\widetilde{\mathbf{R}}(t)=\mathbf{R}(\tau_{\mathbf{R},\mathbf{H}}^{(0,0)}(u)-t)$ and $\widetilde{\mathbf{H}}(t)=-\mathbf{H}(\tau_{\mathbf{R},\mathbf{H}}^{(0,0)}(u)-t)$ give $\widetilde{\mathbf{R}}\stackrel{d}{=}\mathbf{R}$ and $\widetilde{\mathbf{H}}\stackrel{d}{=}\mathbf{H}$ on $[0,\tau_{\mathbf{R},\mathbf{H}}^{(0,0)}(u)]$. This yields\vspace{-0.1cm}
$$-\widetilde{\mathbf{R}}(t)=\widetilde{\mathbf{B}}(t)-(1-d)\widetilde{\mathbf{H}}(t),\qquad t\in[0,\tau_{\mathbf{R},\mathbf{H}}^{(0,0)}(u)],$$
with a minus sign on the left-hand side, so the reversibility is rather subtle. It would be interesting to understand more fully the behaviour of 
$\mathbf{B}$ on intervals 
$[\tau_\mathbf{R}^0(s-),\tau_\mathbf{R}^0(s)]$. E.g., what are the local times of $\mathbf{B}$ on 
$[\tau_\mathbf{R}^0(s-),\tau_\mathbf{R}^0(s)]$, $s\ge 0$?

\bibliographystyle{abbrv}
\bibliography{AldousDiffusion}

\def\polhk#1{\setbox0=\hbox{#1}{\ooalign{\hidewidth\lower1.5ex\hbox{`}\hidewidth\crcr\unhbox0}}}
\begin{thebibliography}{10}

\bibitem{AldousExch}
D.~J. Aldous.
\newblock Exchangeability and related topics.
\newblock In {\em \'{E}cole d'\'et\'e de probabilit\'es de {S}aint-{F}lour,
  {XIII}---1983}, volume 1117 of {\em Lecture Notes in Math.}, pages 1--198.
  Springer, Berlin, 1985.

\bibitem{Bertoin1990}
J.~Bertoin.
\newblock Excursions of a {${\rm BES}_0(d)$} and its drift term {$(0<d<1)$}.
\newblock {\em Probab. Theory Related Fields}, 84(2):231--250, 1990.

\bibitem{Bertoin1990c}
J.~Bertoin.
\newblock Sur une horloge fluctuante pour les processus de {B}essel de petites
  dimensions.
\newblock In {\em S\'{e}minaire de {P}robabilit\'{e}s, {XXIV}, 1988/89}, volume
  1426 of {\em Lecture Notes in Math.}, pages 117--136. Springer, Berlin, 1990.

\bibitem{IPPA}
N.~Forman, S.~Pal, D.~Rizzolo, and M.~Winkel.
\newblock Diffusions on a space of interval partitions: construction from
  marked {L\'e}vy processes.
\newblock arXiv:1909.02584 [math.PR], 2019.

\bibitem{IPPB}
N.~Forman, S.~Pal, D.~Rizzolo, and M.~Winkel.
\newblock Diffusions on a space of interval partitions: {Poisson--Dirichlet}
  stationary distributions.
\newblock arXiv:1910.07626 [math.PR], 2019.

\bibitem{CSP}
J.~Pitman.
\newblock {\em Combinatorial stochastic processes}, volume 1875 of {\em Lecture
  Notes in Mathematics}.
\newblock Springer-Verlag, Berlin, 2006.
\newblock Lectures from the 32nd Summer School on Probability Theory held in
  Saint-Flour, July 7--24, 2002.

\bibitem{PitmYor82}
J.~Pitman and M.~Yor.
\newblock A decomposition of {B}essel bridges.
\newblock {\em Z. Wahrsch. Verw. Gebiete}, 59(4):425--457, 1982.

\bibitem{PitmYorPDAT}
J.~Pitman and M.~Yor.
\newblock The two-parameter {P}oisson--{D}irichlet distribution derived from a
  stable subordinator.
\newblock {\em Ann. Probab.}, 25(2):855--900, 1997.

\bibitem{RevuzYor}
D.~Revuz and M.~Yor.
\newblock {\em Continuous martingales and {B}rownian motion}, volume 293 of
  {\em Grundlehren der Mathematischen Wissenschaften [Fundamental Principles of
  Mathematical Sciences]}.
\newblock Springer-Verlag, Berlin, third edition, 1999.

\bibitem{Salisbury1986a}
T.~S. Salisbury.
\newblock On the it{\^o} excursion process.
\newblock {\em Probability theory and related fields}, 73(3):319--350, 1986.

\end{thebibliography}

\section*{Acknowledgements.} The author would like to thank Jean Bertoin for pointing out parallels between his work \cite{Bertoin1990,Bertoin1990c} and the author's joint work \cite{IPPA,IPPB}, and for several fruitful discussions that led to this project. 
The author would also like to thank his co-authors Noah Forman and Douglas Rizzolo for some feedback on a draft and for further discussions.  

\end{document}